\UseRawInputEncoding
\documentclass[a4,13pt]{amsart}
\usepackage{amssymb,amstext,amsmath,amscd,amsthm,amsfonts,enumerate,graphicx,latexsym, hyperref,cleveref}
\usepackage[usenames]{color}
\usepackage[all]{xy}

 \usepackage{comment}
 \usepackage{mathrsfs}
  \newcommand{\shorteqnote}[1]{ &  & \text{\small\llap{#1}}}
\newtheorem{thm}{Theorem}[section]

\newtheorem{cor}[thm]{Corollary}

\newtheorem{lem}[thm]{Lemma}
\newtheorem{prop}[thm]{Proposition}
\newtheorem{ques}[thm]{Question}
\newtheorem{setup}[thm]{Setup}

\newtheorem{theoremx}{Theorem}
\newtheorem*{theorem*}{Theorem}

\theoremstyle{definition}

\newtheorem{rem}[thm]{Remark}

\newtheorem{ex}[thm]{Example}

\newtheorem*{claim*}{Claim}
\theoremstyle{remark}

\numberwithin{equation}{thm}
\def\cm{\operatorname{CM}}
\def\rf{\operatorname{Ref}}
\def\cond{\mathfrak c}

\def\End{\operatorname{End}}

\def\Hom{\operatorname{Hom}}

\def\m{\mathfrak{m}}

\def\mod{\operatorname{mod}}

\def\syz{\Omega}

\def\tr{\operatorname{tr}}

\def\ad{\operatorname{add}}
\def\gldim{\operatorname{gldim}}

\newcommand{\ses}[3]{0 \to {#1} \to {#2} \to {#3} \to 0}

\begin{document}
%\allowdisplaybreaks
\setlength{\baselineskip}{15pt}
\title{Reflexive modules, self-dual modules and Arf rings}
%\date{September 4, 2014}
\author{Hailong Dao}
\address{Department of Mathematics, University of Kansas, Lawrence, KS 66045-7523, USA}
\email{hdao@ku.edu}
\urladdr{https://www.math.ku.edu/~hdao/}

\thanks{The author is partially supported by Simons Foundation Collaboration Grant 527316}
\dedicatory{Dedicated to Professor Joseph Lipman and his contributions to algebraic geometry and commutative algebra}
\begin{abstract}
We prove a tight connection between reflexive modules over a one-dimensional ring $R$ and its birational extensions that are self-dual as $R$-modules. Consequently, we show that a complete local reduced Arf ring has finitely many indecomposable reflexive modules up to isomorphism, which can be represented precisely by the local rings infinitely near it. We also  characterize Arf rings by the property that any reflexive module is self-dual. We give applications on dimension of subcategories and existence of endomorphism rings of small global dimension (non-commutative desingularizations). Our results indicate striking similarities between Arf rings in dimension one and rational singularities in dimension two from representation-theoretic and categorical perspectives. 
\end{abstract}
\maketitle
%\tableofcontents
%%%%%%%%%%%%%%%%%%%%%%%%%%%%%%%%%%%%%%%%%%%%%%%%%%
%\section*{Convention}

%%%%%%%%%%%%%%%%%%%%%%%%%%%%%%%%%%%%%%%%%%%%%%%%%%%%%%%%%%%%%%%
\section{Introduction}
Let $R$ be a local Cohen-Macaulay ring and let $\rf{R}$ be the category of finitely generated reflexive modules over $R$.  This paper is motivated by the recent work on $\rf{R}$  in \cite{dms}. In particular, we  aim to answer the following question raised there: 

\begin{ques}\cite[Question 7.16]{dms}\label{mainques}
When can we classify all modules in $\rf{R}$? When is $\rf{R}$ of finite representation type (i.e., has finitely many indecomposable objects up to isomorphism)? 
\end{ques}

The analogous questions for $\cm{R}$, the category of maximal Cohen-Macaulay $R$-modules, have been studied extensively and played an important role in the development of the area known as Cohen-Macaulay representation theory. Roughly speaking, for a commutative local  ring $R$, the ``size" of $\cm{R}$, or various subcategories, are intimately and subtly linked to the singularities of $R$.  For a comprehensive survey in this area up to early 2000s, we refer to the  book \cite{leuschke2012cohen}. Perhaps the most famous example of such phenomena comes from  certain algebraic versions of the McKay correspondence, see \cite{BD} for details. Note that over dimension two normal domains, which is the setting of the aforementioned McKay correspondence, $\rf R$ and $\cm R$ coincide.

It turns out that the above question has some very satisfying answers in dimension {\bf one}. They come, rather unexpectedly, from Arf rings, which are illustrious in the study of curves singularity, but have not been considered deeply from the representation-theoretic viewpoint. Following Lipman, \cite{lipman1971stable}, we say that a semi-local, equidimensional, Cohen-Macaulay $R$ of dimension one is Arf if any integrally closed ideal of height one $I$ is {\it  stable}, meaning there is an element $x\in I$ such that $xI=I^2$. The canonical reference for Arf rings is \cite{lipman1971stable}, but recently interest in them has been significantly resurgent among commutative algebraists, see \cite{CCCEGIN, CCGT, EGI, I}.

Recall also that a local ring {\it infinitely near to} $R$ is a local ring of one of the birational extensions of $R$ obtained from repeatedly blowing up the Jacobson radicals (see Section \ref{prelim} for details and references on these concepts). The set of infinitely near local rings of $R$ carries important information on the singularities of $R$ and their desingularization, and  for dimension one it is finite under mild conditions. One of our main results is: 

\begin{theoremx}\label{th1}
Let $R$ be a semi-local, reduced Arf ring of dimension one which is complete with respect to its Jacobson radical. Then the set of indecomposable reflexive $R$-modules is precisely the set of  local rings infinitely near to $R$, up to isomorphism. 

\end{theoremx}

In particular if $R$ is an analytically unramified Arf ring then $\rf{R}$ is of finite representation type, see \Cref{unrami}. This consequence, with certain extra conditions on $R$ and the reflexive modules, have been independently and recently announced in \cite{IKu}. Yet another version of this result, from a different approach using stable trace ideals, will appear in \cite{DaoLindo}.

The key insight behind our \Cref{th1} is a tight connection between reflexive modules over $R$ and over a birational extension $S$ that is self-dual as an $R$-module, see \Cref{refLem1} and \Cref{refLem2}. As Lipman pointed out in his work on Arf rings, even if one cares only about local rings, understanding the semi-local case is crucial since the birational extensions used in various proofs, including the integral closure of $R$, might not be local but only semi-local.

Even in the special case of complete local domains, the finite representation type of $\rf{R}$ does not imply that $R$ is Arf, see \Cref{notArf}. However, our next main result offers a new representation-theoretic  characterization of Arf rings. 

\begin{theoremx}\label{Arfselfdual}
Let $R$ be a semi-local, Cohen-Macaulay, equidimensional ring of dimension one. Assume that the completion of $R$ with respect to its Jacobson radical is reduced (for instance, $R$ is an analytically unramified local ring of dimension one). The following are equivalent:
\begin{enumerate}
\item $R$ is Arf. 
\item Any reflexive $R$-module $M$ is self-dual. That is $M\cong \Hom_R(M,R)$. 
\end{enumerate}
\end{theoremx}

We give some applications of our main results on certain topics, such as dimension of subcategories and global dimension of endomorphism rings, see \Cref{corsyz}, \Cref{dimCM} and \Cref{corNCR}. To us, it is rather interesting that these corollaries show very strong resemblances between Arf rings in dimension one and rational singularities in dimension two, see the paragraph before \Cref{keyques}. Both of these topics were first deeply investigated by Lipman in \cite{lipmanrat, lipman1971stable} and have become heavily studied, yet such similarities have not been noticed before, to the best of our knowledge. Our results also suggest that over one dimensional rings, reflexive modules highly deserve further attention.

The paper is organized as follows. \Cref{prelim} presents notations and key preliminary results on birational extensions and self-dual modules. \Cref{mainproofs} give the proofs of \Cref{th1} and \Cref{Arfselfdual}. Finally, \Cref{apps} contains the main applications, examples and some open questions.\\

\noindent\textbf{Acknowledgements}: The author is grateful to Craig Huneke, Joseph Lipman and Kei-ichi Watanabe for some helpful correspondence on Arf rings and rational singularities, both recently and in the past. The methods and ideas of this paper were heavily inspired from reading Lipman's seminal works \cite{lipmanrat, lipman1971stable} and the author is greatly pleased to discover some new connections between them.

\section{Notations and preliminary results}\label{prelim}

We first fix some notations. Let $R$ be a commutative ring and $M$ a finitely generated $R$-module. Let $M^R$ denote $\Hom_R(M,R)$. Let $f_M$ denote the natural map $M\to M^{RR}$. $M$ is called reflexive if $f_M$ is an isomorphism.  An extension $R\to S$ is called finite birational 
if $S$ is a finitely generated $R$-module and the total rings of fractions are equal, $Q(S) =Q(R)$. The conductor of $S$ in $R$, denoted by $\cond_R(S)$, consists of elements of $R$ that multiply $S$ into $R$.

The trace ideal of $M$, denoted by $\tr(M)$, is the sum $\sum_{f\in \Hom_R(M,R)} f(M)$. When $R$ is local, this ideal is equal to $R$ if and only if $R$ is a direct summand of $M$. Let $\ad M$ denote the category of all summands of (finite) direct sums of $M$. For some background on trace ideals and their utility in studying reflexive modules, we refer to \cite{dms}. 

Next, we focus on the dimension one situation. The background materials are largely taken from \cite{lipman1971stable}. 
Assume that $R$ is semi-local, equidimensional Cohen-Macaulay of dimension one. Let $J_R$ denote the Jacobson radical of $R$, the intersection of all maximal ideals of $R$. Let $I$ be a height one ideal in $R$. The blow-up ring of $I$ in $R$ is defined as $B_R(I):= \cup_{n>0} I^n:_{Q(R)}I^n$. An ideal of $R$ is called stable if $B_R(I)= I:_{Q(R)}I$ (\cite[Definition 1.3]{lipman1971stable}). Equivalently, $I$ is stable if there is an element $x\in I$ such that $xI=I^2$ (cf. \cite[Lemma 1.11]{lipman1971stable}).  $R$ is called an Arf ring if all integrally closed ideal of height one is stable (cf. \cite[Theorem 2.2]{lipman1971stable}). Note that $J_R$ is integrally closed, so it is stable if $R$ is Arf. 

Finally, from $R$ one can build a tower of birational extensions as follows: $A_0=R, A_{i+1} = B_{A_i}(J_{A_i})$ for $i>0$.  If the integral closure $\overline R$ of $R$ in $Q(R)$ is a finite $R$-module, then this chain ends at some $A_l=\overline R$. A local ring $R'$ is called infinitely near to $R$ if it is of the form $R'=(A_i)_P$ for some $A_i$ and a  maximal ideal $P$ of $A_i$.

 We now recall or prove the key preliminary results. 

\begin{thm}\label{faber}
Let $R$ be a Noetherian ring and $M$ be a finite $R$-module. Let $S$ be a finite birational extension of $R$. Consider the following statements.
\begin{enumerate}
    \item $M$ is a module over $S$.
    \item $\tr(M)\subseteq \cond_R(S)$.
\end{enumerate}
Then $(1)$ implies $(2)$. The converse is true if $M$ is a reflexive $R$-module. 

\end{thm}

\begin{proof}
See \cite[Theorem 2.9]{dms}.
\end{proof}

\begin{lem}\label{M**}
If $M$ is a finitely generated module and $M\cong M^{RR}$ then $M$ is reflexive. 
\end{lem}

\begin{proof}
Let $N=M^R$. Since $f_{N^R}: N^R \to N^{RRR} $ splits (dualizing the natural map $f_N$ gives the opposite map), it follows that the map $f:M\to M^{RR}$ splits. But as $M\cong M^{RR}$ and $M$ is Noetherian, $f$ is an isomorphism. 
\end{proof}

\begin{prop}\label{birHom}
Let $R\to S$ be a finite birational extension of one-dimensional Cohen-Macaulay rings. If $M,N$ are maximal Cohen-Macaulay $S$-modules, then $\Hom_R(M,N)=\Hom_S(M,N)$. 
\end{prop}

\begin{proof}
See \cite[4.14]{leuschke2012cohen}, note that the proof there does not use the local assumption.
\end{proof}

\begin{prop}\label{proj}
Let $R$ be a semi-local ring and $M$ a finitely generated projective module. Then $M$ is self-$R$-dual, that is $M\cong M^R$.

\end{prop}

\begin{proof}
We can write $R$ as product of semi-local rings $R_i$s, each of them has connected spectrum. On each $R_i$, $M$ is finite free, thus the assertion follows. 
\end{proof}

The following is our key technical result, which relates reflexive modules over $R$ to the reflexive modules over self-dual birational extensions. 

\begin{lem}\label{refLem1}
Let $R\to S$ be a finite birational extension of one-dimensional Cohen-Maculay rings such that $\Hom_R(S,R) \cong_R S$. Let $M$ be a maximal Cohen-Macaulay $S$-module. Then $\Hom_R(M,R)\cong \Hom_S(M,S)$, both as $R$ and $S$ modules.  In particular, a $S$-module  is $R$-reflexive  if and only if it is $S$-reflexive. 
\end{lem}

\begin{proof}
Let $M$ be an $S$-module. Recall the notation  $M^R = \Hom_R(M,R)$. We have:
\begin{flalign*} 
& M^R = \Hom_R(M,R)  \\
& = \Hom_R(M\otimes_SS,R)  \\
&\cong \Hom_S(M,S^R)  \shorteqnote{(Hom-tensor adjointness)}\\ 
&= \Hom_R(M,S^R)   \shorteqnote{(by \Cref{birHom} )} \\
&\cong_R \Hom_R(M,S)  \shorteqnote{(by assumption)}\\
&= \Hom_S(M,S) \shorteqnote{(by \Cref{birHom})}\\
& = M^S
\end{flalign*}

\Cref{birHom} then implies that $M^R\cong_S M^S$ as well. Apply this twice gives $M^{RR} \cong M^{SS}$ both as $R$ and $S$ modules.  Finally  \Cref{M**} completes the proof. 
\end{proof}

\begin{rem}
The condition $\Hom_R(S,R) \cong_R S$ is rather weak. Without the birational assumption, \Cref{refLem1} will fail spectacularly. Take $R$ to be a DVR and $S$ a Cohen-Macaulay but not Gorenstein finite extension of $R$. Then $S\cong_RR^r$ for some $r>1$, so $\Hom_R(S,R) \cong_R S$. Every CM $S$-module is $R$-free and hence $R$-reflexive, but they can't all be $S$-reflexive as $S$ is not Gorenstein.  
\end{rem}

\begin{thm}\cite[Theorem 2.2]{lipman1971stable}\label{LipThm}
$R$ is Arf if and only if any local ring $B$ infinitely near $R$ has minimal multiplicity (the Hilbert-Samuel multiplicity of $B$ is equal to its embedding dimension). 
\end{thm}

\begin{prop}\label{stablechar}
Let $R$ be a semi-local, Cohen-Macaulay, equidimensional ring of dimension one. Let $J=J_R$ and $S=J:_{Q(R)}J=\End_R(J)$. The following are equivalent:
\begin{enumerate}

\item $J$ is stable.
\item $J=xS$ for some (regular) $x\in J$.
\item $J\cong S$ as $R$-modules. 
\item $J$ is self-dual as an $R$-module ($J\cong J^R=\Hom_R(J,R)$).  
\item $S$ is self-dual as an $R$-module. 
\end{enumerate}

\end{prop}

\begin{proof}
We first prove $(1)\iff (2)\iff (3)$. By \Cref{birHom}, $J\cong_R S$ if and only if $J \cong_S S$ if and only if $J=xS$ for some $x\in S$ if and only if $J=xS$ for some $x\in J$, and such $x$ must be regular.  Thus $(2)\iff (3)$. If $(2)$ holds, then $x^{-1}J=S$ is a ring, so $J$ is stable by \cite[Lemma 1.11]{lipman1971stable}. If $J$ is stable, then $B_R(J)=S$ and hence $J=JS=xS$ for some $x\in S$, by \cite[Proposition 1.1, Definition 1.3]{lipman1971stable}.

Finally, note that $S\cong \Hom_R(J,J) \cong \Hom_R(J,R)$ and $J$ is reflexive. Both of these statements hold because  at a maximal ideal $\m$, $J_{\m}=\m R_{\m}$ and such statements holds for the maximal ideal of a local Cohen-Macaulay one-dimensional ring, see \cite[Corollary 3.2]{dms}.  Thus $J^R\cong S$ and $S^R\cong J$, and so evidently $(3)\iff (4)\iff (5)$. 
\end{proof}

\begin{cor}\label{corinduct}
Let $R$ be a semi-local, Cohen-Macaulay, equidimensional ring of dimension one. Let $J=J_R$ and $S=\End_R(J)$. The following are equivalent:
\begin{enumerate}
\item $R$ is Arf.
\item $J$ is stable and $S$ is Arf. 
\item $S$ is self-dual as an $R$-module ($S\cong S^R$) and $S$ is Arf. 
\end{enumerate}
\end{cor}

\begin{proof}
If $R$ is Arf, then $J$ is stable and $S$ is Arf since any local ring infinitely near $S$ is also infinitely near $R$. Assume $(2)$, then we just need to show that any local ring $B=R_P$ with $P$ a maximal ideal of $R$ has minimal multiplicity. But $P_P=J_P$ is stable, and  $R_P$ has minimal multiplicity by \cite[Corollary 1.10]{lipman1971stable}.
\end{proof}

\section{Proofs of main results}\label{mainproofs}

In this section we prove the main theorems in the introduction. First, we describe some convenient setups.

\begin{setup}\label{setup}
Let $R$ be a semilocal, equidimensional, Cohen-Macaulay ring of dimension one. Let $J=J_R$ be the Jacobson radical of $R$ and we assume that $R$ is complete with respect to $J$. 
\end{setup}

\begin{setup}\label{setup2}
Assume \Cref{setup}. Further assume that $R$ is reduced. Then the integral closure $\overline R$ of $R$ is a finite $R$-module, and there exist a finite chain of birational extensions $R=A_0 \subsetneq A_1 \subsetneq \dots \subsetneq A_l = \overline R$ with $A_{i+1} =B_{A_i}(J_{A_i})$. Note that if $R$ is Arf, then $A_{i+1}=\End_{A_i}(J_{A_i})$ for each $i$ since  $J_{A_i}$ is $A_i$-stable.

\end{setup}

\begin{lem}\label{refLem2}
Assume \Cref{setup}. Assume further that $J$ is stable. Let $S=\End_R(J)$. Then $\rf(R) = \ad(R) \cup \rf(S)$.
\end{lem}

\begin{proof}
Projective $R$ modules are reflexive, so \Cref{refLem1} implies that $\ad(R) \cup \rf(S)\subset \rf(R)$ (note that $S$ is self-dual by \Cref{stablechar}). Take any indecomposable $M\in \rf(R)$. Let $\{\m_1,..., \m_s\}$ be the maximal ideals of $R$. Note that as $R$ is complete with respect to $J$, it is equal to the direct product of the localizations $R_i=R_{\m_i}$. Hence if $\tr(M)$ is not contained in some $\m_i$, then $R_i  \in \ad(M)$, so $M=R_i$. Otherwise, $\tr(M)\subset J$, so $M\in \cm(S)$ by \Cref{faber}. Then \Cref{refLem1} again shows that $M\in \rf(S)$, and we are done. 
\end{proof}

We now prove \Cref{th1}. 

\begin{thm}\label{main_thm}
Assume  \Cref{setup2}. Further assume that $R$ is an Arf ring.  Then $\rf(R) = \ad(\bigoplus_0^l A_i)$.  More precisely, each indecomposable reflexive $R$-module is isomorphic to one of the  local rings infinitely near to $R$.
\end{thm}

\begin{proof}
Let $J=J_R$ and $A_1=\End_R(J)$. Since $R$ is Arf, $J$ is stable.  Then \Cref{refLem2} shows that $\rf(R) = \ad(R)\cup \rf(A_1)$. As each $A_i$ is Arf by \Cref{LipThm} , we can repeat this process to arrive at the conclusion. Since each $A_i$ is complete, they are direct sums of their local rings at maximal ideals, which are precisely the local rings infinitely near to $R$.
\end{proof}

\begin{cor}\label{unrami}
Let $R$ be a local, analytically unramified Arf ring. Then there exist a finitely generated module $M$ such that $\rf{R}\subset \ad{M}$. 
\end{cor}

\begin{proof}
Let $\hat R$ be the completion of $R$ with respect to $\m_R$. Let $X$ be the direct sum of all local rings infinitely near $\hat R$. Then there is an $R$-module $M$ such that $X\in \ad{\hat M}$, see for instance \cite[Corollary 3.6]{stable}. It follows that for any reflexive $R$-module $N$, $\hat N \in \ad{\hat M}$, which implies $N\in \ad M$ (see e.g. \cite[Corollary 1.15]{leuschke2012cohen}).  
\end{proof}

We end this section by providing the proof of \Cref{Arfselfdual}. 

\begin{proof}{(of \Cref{Arfselfdual})} First assume $R$ is Arf. To prove $(2)$, we can replace $R$ by its completion by $J_R$, which is still Arf by \cite[Corollary 2.7]{lipman1971stable}. Thus, we are now in \Cref{setup2}. By \Cref{main_thm}, each indecomposable reflexive module is isomorphic to an indecomposable summand $B$ of $A_i$ for some $i$. Then $B$ is self-dual as an $A_i$ module by \Cref{proj}, and thus is self-dual with respect to $R$ by \Cref{refLem1}. 

Now assume $(2)$. Note that the integral closure $\overline R$ of $R$ is module-finite by assumptions, and we will prove by induction on $n(R):=\ell(\bar R/R)$. If $n(R)=0$ then $R$ is regular, and we are done. Else, let $J=J_R$ and $S=J:_{Q(R)}J$. As $J$ is reflexive, it is self-dual, and so is $S\cong J^R$ (see the proof of \Cref{stablechar}). By \Cref{refLem1},  $\rf{S}\subset \rf{R}$ and all reflexive $S$-module are self-dual with respect to  $S$. As $n(S)<n(R)$, $S$ is Arf by induction, and thus $R$ is Arf by \Cref{corinduct}. 

\end{proof}

\section{Further consequences, examples and questions}\label{apps}
We begin this section with some examples of Arf rings. 

\begin{ex}
Let $k$ be a field and $R=k[[x,y]]/(xy)$ with $\m=(x,y)$. Then as in \Cref{main_thm}, $\overline R = A_1= \End_R(\m) = R/(x)\times R/(y)$. The indecomposable reflexive modules are $R, R/(x), R/(y)$ and they are all self-dual. 
\end{ex}

\begin{ex}
Let $k$ be a field. Fix an integer $e\geq 2$. For each $i\geq 0$ let $N_i$ be the numerical semigroup generated by $\{e, ie+1, ie+2,..., ie+{e-1}\}$. Let $R_i = k[[t^a, a\in H_i]]$. Then if $R= R_n$ for some $n\geq 0$, the $i$-th successive blow-up $A_i$ of $R$ (as in \Cref{setup2})  will be $R_{n-i}$. Each of them is local with minimal multiplicity $e$, so $R$ is Arf and $\rf R=\ad \oplus_{0\leq j\leq n} R_j$. 
\end{ex}

Our results have interesting consequences on various related topics. Here we sketch some immediate ones and raise a few obvious questions inspired by them.

First, it would be fun to extend  \Cref{Arfselfdual} to other Krull dimensions. 
\begin{ques}
Can we characterize rings $R$ over which any reflexive module is self $R$-dual?
\end{ques}

Next, even in dimension one, \Cref{mainques} has not been completely resolved, as can be seen in the below example. Thus a complete description of rings of finite reflexive type is still unknown. 

\begin{ex}\label{notArf}
Not all rings of finite reflexive type are Arf rings. An example is $R= \mathbb C[[t^3,t^5]]$ which does not have minimal multiplicity, so can not be Arf. Yet since $R$ is Gorenstein and finite CM type, it has finite reflexive type. Or we can note that as $\End_R(\m_R) = \mathbb C[[t^3,t^5,t^7]]$ has finite CM type, $R$ has finite reflexive type by \cite[Proposition 6.12]{dms}.
\end{ex}

Recall now that when $R$ is a one-dimensional reduced local ring, then $\rf R$ is the same as $\syz\cm R$, the category of syzygies of maximal Cohen-Macaulay $R$-modules. So \Cref{main_thm} gives:

\begin{cor}\label{corsyz}
If $R$ is an analytically unramified Arf ring then $\syz\cm R$ has finite representation type. 
\end{cor}
 
Next, we recall that various notions of ``sizes" of a subcategories of modules or derived categories over a scheme $X$ have been introduced and shown to be deeply related to its singularities. For algebraic references see \cite{rou, DT1, DT2} and especially the introduction of  \cite{Nee}. One of them is the notion of a dimension of a subcategory of $\mod R$, as defined in \cite[Definition 3.3]{DT2}. 

\begin{cor}\label{dimCM}
If $R$ is an analytically unramified Arf ring then the dimension of $\cm R$, in the sense of  \cite[Definition 3.3]{DT2}, is at most one. 
\end{cor}
 
\begin{proof}
In general, if $\syz\cm R$ is of finite type then $\dim \cm R\leq 1$, see the proof of \cite[Proposition 3.7]{DT2}. So the assertion follows from \Cref{corsyz}.
\end{proof} 

In recent years, the existence of endomorphism rings of finite global dimension, also known as ``non-commutative desingularizations" have attracted a lot of attention, see for instance \cite{Dao, DFI,  DIITVY, DITV,  FMK, IW, SV, V}. Our next corollary states that over Arf rings, one can construct many endomorphism rings of small global dimensions. 

\begin{cor}\label{corNCR}
Let $R$ be a complete, reduced, local Arf ring. Let $\{A_i\}_{i\geq 0}$ be as in \Cref{setup2}. Let  $X_i$ be the direct sum of all local rings infinitely near $A_i$ and $E_i= \End_R(X_i)$. Then the global dimension of each $E_i$ is at most $2$.  
\end{cor}

\begin{proof}
We first consider $X_0$, by \Cref{main_thm} $\ad X_0=\rf R= \syz\cm R$ and $E_0=\End_R(X_0)$. By Auslander's projectivization, any second syzygy module $U$ over $E_0$ has the form $U=\Hom_R(X_0, M)$, where $M$ fits into a short exact sequence $\ses{M}{X_0^n}{N}$ with $N\in \cm R$ (see for instance the proof of \cite[Proposition 2.11]{IW}). Thus $M\in \syz\cm R =\ad X_0$ (\cite[Lemma 2.14]{IW}), so $U$ is projective over $E_0$. Thus $\gldim(E_0)\leq 2$.  By \Cref{birHom}, $E_i = \End_R(X_i)= \End_{A_i}(X_i)$ and $\ad X_i=\rf A_i$, so we are done by the first case. 
\end{proof}

Our previous results, especially when compared with \cite[Proposition 3.7]{DT2} and \cite[Proposition 2.11]{IW}, establish remarkable similarities between Arf rings and two-dimensional rational singularities from representation-theoretic and categorical points of view. Thus, it is natural to ask:

\begin{ques}\label{keyques}
Are there more direct relationships between Arf rings and local rings with (isolated) rational singularity of Krull dimension two? For instance, if $R$ is a complete normal domain of dimension $2$ with algebraically closed residue field and rational singularity, can we always find a regular element $x \in \m_R$ such that $R/xR$ is  Arf?
\end{ques}

\end{document}